\documentclass{amsart}
\usepackage{amssymb, amsmath, latexsym, mathabx, dsfont,tikz}
\usepackage{multirow}
\usepackage{setspace}
\usepackage{enumerate}

\renewcommand{\baselinestretch}{\baselinestretch}
\renewcommand{\baselinestretch}{1.1}
\numberwithin{equation}{section}

\newtheorem{thm}{Theorem}[section]
\newtheorem{lem}[thm]{Lemma}

\theoremstyle{definition}

\theoremstyle{remark}
\newtheorem{rmk}[thm]{Remark}

\numberwithin{equation}{section}

\newcommand{\gen}{\text{gen}}

\newcommand{\z}{{\mathbb Z}}

\newcommand{\Mod}[1]{\ (\mathrm{mod}\ #1)}

\begin{document}
\title[Representations by a ternary sum of triangular numbers]{The number of representations by a ternary sum of triangular numbers}

\author{ Mingyu Kim and Byeong-Kweon Oh}

\address{Department of Mathematical Sciences, Seoul National University, Seoul 151-747, Korea}
\email{kmg2562@snu.ac.kr}

\address{Department of Mathematical Sciences and Research Institute of Mathematics, Seoul National University, Seoul 151-747, Korea}
\email{bkoh@snu.ac.kr}

\thanks{This work of the second author was supported by the National Research Foundation of Korea (NRF-2017R1A2B4003758).}

\subjclass[2000]{Primary 11E12, 11E20} \keywords{Representations of ternary quadratic forms, squares}


\begin{abstract} For positive integers $a,b,c$, and an integer $n$, the number of integer solutions $(x,y,z) \in \z^3$ of  $a \frac{x(x-1)}{2} + b \frac{y(y-1)}{2} + c \frac{z(z-1)}{2} = n$ is denoted by $t(a,b,c;n)$. In this article, we prove some relations between $t(a,b,c;n)$ and the numbers of representations of integers by some ternary quadratic forms. In particular, we prove various conjectures given by Z. H. Sun in \cite{s}.         
\end{abstract}

\maketitle

\section{Introduction}
For a positive integer $x$, a non negative integer of the form $T_x=\frac{x(x-1)}2$ is called a {\it triangular number}. For example, $0,1,3,6,10,15,\dots$ are triangular numbers. Since $T_x=T_{1-x}$, $T_x$ is a triangular number for any integer $x$. For positive integers $a_1,a_2,\dots a_k$, a polynomial of the form 
$$
\mathcal T_{(a_1,\dots,a_k)}(x_1,\dots,x_k)=a_1T_{x_1}+a_2T_{x_2}+\cdots+a_kT_{x_k}
$$
is called a $k$-ary sum of triangular numbers. For a non negative integer $n$, we define 
$$
T(a_1,\dots,a_k;n)=\left\{(x_1,\dots,x_k) \in \z^k : \mathcal T_{(a_1,\dots,a_k)}(x_1,\dots,x_k)=n \right\}
$$ 
and $t(a_1,\dots,a_k;n)=\vert T(a_1,\dots,a_k;n)\vert$. One may easily show that 
$$
t(a_1,\dots,a_k;n)=\vert\{ (x_1,\dots,x_k) \in (\z_o)^k : a_1x_1^2+\cdots+a_kx_k^2=8n+a_1+\cdots+a_k\}\vert,
$$   
where $\z_o$ is the set of all odd integers. Hence $t(a_1,\dots,a_k;n)$ is closely related with the number of representations by some diagonal quadratic form of rank $k$. For example, if $k=3$ and $a_1=a_2=a_3=1$, then every integer solution $(x,y,z)$ of $x^2+y^2+z^2=8n+3$ is in $(\z_o)^3$. Therefore, for any positive integer $n$, we have  
$$
t(1,1,1;n)=\vert\{ (x,y,z) \in \z^3 : x^2+y^2+z^2=8n+3\}\vert=24H(-(8n+3)),
$$
where $H(-D)$ is the Hurwitz class number with discriminant $-D$. For various results in this direction, see \cite{ach}, \cite{bch}, \cite{s0} and \cite{y}. 

 Recently, Sun proved in \cite{s} various relations between $t(a_1,\dots,a_k;n)$ and the numbers of representations of integers by some diagonal quadratic forms. He also conjectured various relations between $t(a,b,c;n)$ and the numbers of representations by some ternary diagonal quadratic forms. 

In this article, we consider the number $t(a,b,c;n)$ of representations by a ternary sum of triangular numbers. 
We show that  for any positive integers  $a,b,c$ such that $(a,b,c)=1$,  $t(a,b,c;n)$ is equal to the number of representations of a subform  of the ternary diagonal quadratic form $ax^2+by^2+cz^2$, if $a+b+c$ is not divisible by $8$, or a difference of the numbers of representations of two ternary quadratic forms otherwise. 

In Section 3, we prove all conjectures in \cite{s} on ternary sums of triangular numbers, which are Conjectures 6.1$\sim$6.4 and 6.7. In fact, we generalize Conjectures 6.1 and 6.2 in \cite{s}, and prove these generalized conjectures.  Note that Conjectures 6.5 and 6.6 in \cite {s} are about quaternary sums of triangular numbers, which will be treated later.       

An integral quadratic form  $f(x_1,x_2,\dots,x_k)$ of rank $k$ is a degree $2$ homogeneous polynomial
$$
f(x_1,x_2,\dots,x_k)=\sum_{1\le i,j\le k} a_{ij}x_ix_j \quad (a_{ij}=a_{ji} \in \z).
$$
We always assume that $f$ is positive definite, that is, the corresponding symmetric matrix $(a_{ij}) \in M_{k\times k}(\z)$ is positive definite. 
If $a_{ij}=0$ for any $i\ne j$, then we simply write $f=\langle a_{11},\dots,a_{kk}\rangle$. For an integer $n$, if the Diophantine equation $f(x_1,x_2,\dots,x_k)=n$ has an integer solution, then we say $n$ is {\it represented by $f$}. We define
$$
R(f,n)=\{(x_1,\dots,x_k) \in \mathbb Z^k : f(x_1,\dots,x_k)=n\},
$$ 
and $r(f,n)=\vert R(f,n)\vert$. Since we are assuming that $f$ is positive definite, the above set is always finite.
The genus of $f$, denoted by $\text{gen}(f)$, is the set of all quadratic forms that are locally isometric to $f$. The number of isometry classes in $\text{gen}(f)$ is called the class number of $f$.

Any unexplained notations and terminologies on integral quadratic forms can be found in \cite{Ki} or \cite{OM}. 

\section{Representations of ternary sums of triangular numbers}

Let $a,b$ and $c$ be positive integers such that $(a,b,c)=1$. Throughout this section, we assume, without loss of generality, that $a$ is odd. We show that the number $t(a,b,c;n)$ is equal to the number of representations of a subform  of the ternary diagonal quadratic form $ax^2+by^2+cz^2$, if $a+b+c$ is not divisible by $8$, or a difference of the numbers of representations of two ternary quadratic forms otherwise. 

Let $f(x,y,z)=ax^2+by^2+cz^2$ be a ternary diagonal quadratic form. Recall that 
$$
t(a,b,c;n)=\vert\{ (x,y,z) \in \z^3 : f(x,y,z)=n, \ xyz\equiv 1 \Mod 2\}\vert.
$$

\begin{lem} \label{odd} Assume that $a+b+c$ is odd. For any positive integer $n$, we have 
$$
t(a,b,c;n)=r(f(x,x-2y,x-2z),8n+a+b+c).
$$
In particular, if $a\equiv b \equiv c \Mod 4$, then  we have
$$
t(a,b,c;n)=r(f(x,y,z), 8n+a+b+c).
$$
\end{lem}
\begin{proof} Let $g(x,y,z)=f(x,x-2y,x-2z)$. Define a map $\phi : T(a,b,c;n) \to R(g,n)$ by $\phi(x,y,z)=(x,\frac{x-y}2,\frac{x-z}2)$. Then one may easily show that it is a bijective map. 

  Now, assume that  $a\equiv b \equiv c \Mod 4$. If $ax^2+by^2+cz^2=8n+a+b+c$ for some integers $x,y$ and $z$, then one may easily show that $x,y$ and $z$ are all odd. The lemma follows directly  from this. 
\end{proof}

\begin{lem} \label{even} Assume that $S=a+b+c$ is even, and without loss of generality, we also assume that both $a$ and $b$ are odd and $c$ is even. Then, for any positive integer $n$, we have 
$$
t(a,b,c;n)\!=\!\begin{cases} r(f(x,y,z),8n+S)\!\!\!\!  &\text{if $S \equiv 2 \Mod 4$ and $c \equiv 4 \Mod 8$},\\ 
                         r(f(x,y,y-2z),8n+S)\! \!\!\! &\text{if $S \equiv 2 \Mod 4$ and $c \not \equiv 4 \Mod 8$},\\ 
                          2r(f(x,x-4y,z),8n+S)\! \!\!\! &\text{if $S \equiv 4 \Mod 8$ and $c  \equiv 2 \Mod 4$},\\ 
                          2r(f(x,x-4y,x-2z),8n+S)\! \!\!\!\! &\text{if $S \equiv 4 \Mod 8$ and $c \equiv 0 \Mod 4$},\\ 
\end{cases}                          
$$
and if $S \equiv 0 \Mod 8$, then 
$$
t(a,b,c;n)=r(f(x,x-2y,x-2z),8n+S)-r\left(f(x,y,z),2n+\displaystyle\frac{S}4\right). 
$$  
\end{lem}

\begin{proof} Since the proof is quite similar to each other, we only provide the proof of the fourth case, that is, the case when $S \equiv 4 \Mod 8$ and $c \equiv 0 \Mod 4$.  Let $g(x,y,z)=f(x,x-4y,x-2z)$. We define a map
$$
\begin{array} {rl}  
&\psi : \{ (x,y,z) \in (\z_o)^3  : f(x,y,z)=8n+S, \ x \equiv y \Mod 4\} \\ [0.5em]
 & \hskip 2pc  \to \{ (x,y,z) \in \z^3  : g(x,y,z)=8n+S\} \ \  \text{by} \ \ \psi(x,y,z)=\left(x,\displaystyle \frac {x-y}4,\frac{x-z}2\right).
\end{array}
$$   
From the assumption, it is well defined.  Conversely, assume that $g(x,y,z)=8n+S$ for some $(x,y,z) \in \z^3$. 
Since
$$
f(x,x-4y,x-2z)=ax^2+b(x-4y)^2+c(x-2z)^2\equiv ax^2+bx^2+cx^2\equiv Sx^2 \equiv S \Mod 8
$$   
and $S \equiv 4 \Mod 8$, the integer $x$ is odd.  Therefore, the map $(x,y,z) \to (x,x-4y,x-2z)$ is an inverse map of $\psi$.
The lemma follows from this and the fact that 
$$
t(a,b,c;n)=2\vert\{ (x,y,z) \in (\z_o)^3  : f(x,y,z)=8n+S, \ x \equiv y \Mod 4\}\vert.
$$
This completes the proof.  \end{proof}
\section{Sums of triangular numbers and diagonal quadratic forms}

In this section, we generalize some conjectures  given by Sun in \cite{s} on the relations between $t(a,b,c;n)$ and  the numbers of representations of integers by some ternary quadratic forms, and prove these conjectures. 
 
Let $f(x_1,x_2,\dots,x_k)$ be an integral quadratic form of rank $k$ and let $n$ be an integer. For a vector $\bold d=(d_1,\dots,d_k) \in (\z/2\z)^k$, we  define
$$
R_{\bold d}(f,n)=\{ (x_1,\dots,x_k) \in R(f,n) : (x_1,\dots,x_k) \equiv (d_1,\dots,d_k) \ (\text{mod} \ 2)\}.
$$
The cardinality of the above set will be denoted by $r_{\bold d}(f,n)$. Note that
$$
t(a,b,c;n)=r_{(1,1,1)}(ax^2+by^2+cz^2,8n+a+b+c). 
$$
We also define
$$
\widetilde R_{(1,1)}(ax^2+by^2,N)=\{ (x,y) \in  R_{(1,1)}(ax^2+by^2,N) : x \not \equiv y \ (\text{mod } 4)\}.
$$
Note that if we define the cardinality of $\widetilde R_{(1,1)}(ax^2+by^2,N)$ by $\widetilde{r}_{(1,1)}
(ax^2+by^2,N)$, then we have
$$
r_{(1,1)}(ax^2+by^2,N)=2\cdot \widetilde{r}_{(1,1)}(ax^2+by^2,N).
$$

\begin{lem} \label{lem1} Let $m$ be a positive integer. 
\begin{itemize}
\item [(i)]
If $m \equiv 1 \Mod{4}$, then we have 
$$
2r_{(1,0)}(x^2+3y^2,m)=r_{(1,1)}(x^2+3y^2,4m).
$$
\item [(ii)]
If $m \equiv 3 \Mod{4}$, then we have 
$$
2r_{(0,1)}(x^2+3y^2,m)=r_{(1,1)}(x^2+3y^2,4m).
$$
\item [(iii)]
If $m \equiv 4 \Mod{8}$, then we have 
$$
2r_{(0,0)}(x^2+3y^2,m)=r_{(1,1)}(x^2+3y^2,m).
$$
\end{itemize}
\end{lem}

\begin{proof} (i) Note that the map 
$$
\psi_1 : R_{(1,0)}(x^2+3y^2,m) \to 
\widetilde{R}_{(1,1)}(x^2+3y^2,4m) \ \ 
\text{defined by}  \ \ \psi_1(x,y)=(x+3y,-x+y)
$$
 is a bijective map. 
 
\noindent (ii) If we define a map 
$$
\psi_2 : R_{(0,1)}(x^2+3y^2,m) \to 
\widetilde{R}_{(1,1)}(x^2+3y^2,4m)  \ \ 
\text{by}  \ \  \psi_2(x,y)=(x+3y,-x+y),
$$ then one may easily check that it is a  bijective map. 

\noindent (iii) One may easily show that if we define a map
$$
\psi_3 :{R}_{(0,0)}
(x^2+3y^2,m) \to \widetilde R_{(1,1)}(x^2+3y^2,m)
\ \ 
\text{by} \ \   \psi_3(x,y)=\displaystyle \left(\frac {x+3y}2,\frac{-x+y}2\right),
$$
then it is a bijective map. 
\end{proof}


\begin{lem} \label{lem2} Let $a,b \ (a<b)$ be positive odd integers such that $\gcd(a,b)=1$ and 
$a+b \equiv 0 \Mod{8}$. Then  
\begin{equation} \label{complete}
r_{(1,1)}(ax^2+by^2,m)=r_{(1,1)}(ax^2+by^2,4m)
\end{equation}
for any integer $m$ divisible by $8$ if and only if $(a,b) \in \{ (3,5), (1,7), (1,15) \}$.
\end{lem}

\begin{proof}
Assume that Equation \eqref{complete} holds for any integer $m$ divisible by $8$. 
Let $a+b=2^u k$ for some integer $u \ge 3$ and an odd integer $k$. Note that $1 \le a < 2^{u-1} k$. 

First, we assume $u \ge 5$. 
Since 
$$
a\cdot 1^2+(2^u k-a)\cdot 1^2=4\cdot 2^{u-2} k  \quad \text{and} \quad 2^{u-2}k \equiv 0 \Mod{8},
$$ 
there exist odd integers $x$ and $y$ satisfying 
$ax^2+(2^u k-a)y^2=2^{u-2}k$, which is a contradiction.

Next, assume that $u=4$. 
Since 
$$
a \cdot 7^2+(16k-a)\cdot 1^2=4(4k+12a) \quad \text{and} \quad 4k+12a \equiv 0 \Mod{8},
$$ 
there exist two odd integers $x_1,y_1$ such that $ax_1^2+(16k-a)y_1^2=4k+12a$.
Thus, $4k+12a \ge 16k$ and hence $k \le a$. 
Now, since $a\cdot 1^2 +(16k-a)\cdot 1^2=16k$,
there are two positive odd integers $x_2,y_2$ with $ax_2^2+(16k-a)y_2^2=64k$. 
Since $16k-a>8k$ by assumption, we have $y_2^2=1$. 
Furthermore, since $ax_2^2=a+48k \le 49a$, $(x_2,a)=(3,6k), (5,2k)$ or $(7,k)$.
Since $a$ is odd,  we have $(a,b)=(1,15)$ in this case. 

Finally, we assume that $u=3$. 
Since $a\cdot 1^2+(8k-a)\cdot 1^2 =8k$, there are positive odd integers $x_3,y_3$ such that 
$ax_3^2+(8k-a)y_3^2=32k$. 
Hence we have
\begin{equation} \label{2.2}
y_3^2=1 \quad \text{and} \quad ax_3^2=a+24k.
\end{equation}
 Note that if $x_3=3$, then $(a,b)=(3,5)$ and if $x_3=5$, then $(a,b)=(1,7)$. 
Assume that $x_3 \ge 7$, that is, $2a \le k$. 
Since $a\cdot 3^2+(8k-a)\cdot 1^2=8k+8a$, there are two odd integers $x_4,y_4$ such that 
$ax_4^2+(8k-a)y_4^2=32k+32a$. 
If $y_4^2 \ge 9$, then $a+72k-9a \le 32k+32a$, 
which is a contradiction to the assumption that $2a \le k$. 
Hence we have 
\begin{equation} \label{2.3}
y_4^2=1\quad \text{and} \quad ax_4^2=33a+24k.
\end{equation}
Now, by Equations \eqref{2.2} and \eqref{2.3}, we have $x_4^2-x_3^2=32$. 
Therefore, $x_3^2=49$, $x_4^2=81$, and $k=2a$.  
which is a contradiction to the assumption that $k$ is odd. 

To prove the converse, 
we define three maps 
$$
\chi_1 : \widetilde{R}_{(1,1)}(3x^2+5y^2,m) \to \widetilde{R}_{(1,1)}
(3x^2+5y^2,4m)
\ \  
\text{by}  \ \ \chi_1(x,y)=\displaystyle\left(\frac {x-5y}2,\frac{3x+y}2\right),
$$
$$
\chi_2 : \widetilde{R}_{(1,1)}(x^2+7y^2,m) \to 
\widetilde{R}_{(1,1)}(x^2+7y^2,4m)
\ \  \text{by} \ \ \chi_2(x,y)=\displaystyle\left(\frac {3x-7y}2,
\frac{x+3y}2\right),
$$ and
$$
\chi_3 : \widetilde{R}_{(1,1)}(x^2+15y^2,m) \to 
\widetilde{R}_{(1,1)}(x^2+15y^2,4m)
\ \
 \text{by}  \ \ \chi_3(x,y)=\displaystyle\left(\frac {x+15y}2,\frac{-x+y}2\right).
$$
One may easily show that the above three maps are all bijective. 
\end{proof}


\begin{thm} \label{thm1} Let $a,b,c$ be positive integers such that $(a,b,c)\ne(1,1,1)$ and $\gcd(a,b,c)=1$.  Assume that two 
of three fractions $\frac{b}{a}, \frac{c}{b}, \frac{c}{a}$ are contained in $\left\{ 1,\frac{5}{3},7,15 \right\}$. 
Then, for any positive integer $n$, we have  
$$
2t(a,b,c;n)=r(ax^2+by^2+cz^2,4(8n+a+b+c))-r(ax^2+by^2+cz^2,8n+a+b+c).
$$
\end{thm}
\begin{proof} Note that all of $a,b$ and $c$ are odd. Furthermore,
from the assumption, one may easily show that 
$$
-a \equiv b \equiv c \ (\text{mod} \ 8), \quad a \equiv -b \equiv c \ (\text{mod} \ 8) \quad \text{or} \quad a \equiv b \equiv -c \ (\text{mod} \ 8).
$$  
By switching the roles of $a, b$ and $c$ if necessary,  we may assume $a \equiv b \equiv -c \ (\text{mod} \ 8) $. Then we have
$$
\left(\frac a{(a,c)},\frac c{(a,c)}\right), \left(\frac b{(b,c)},\frac c{(b,c)}\right) \in \{ (3,5), (5,3), (1,7), (7,1),(1,15), (15,1)\}.
$$

Let 
$$
f=f(x,y,z)=ax^2+by^2+cz^2 \quad \text{and} \quad N=8n+a+b+c.
$$ 
One may easily show that if $f(x,y,z)=4N$, then 
$$
\left( ax^2, by^2, cz^2 \right) \equiv (0,0,4), (0,4,0), (a,4,c), (4,0,0), (4,b,c), \ \ \text{or} \ \ (4,4,4) \ (\text{mod} \ 8).
$$
Let 
\begin{eqnarray*}
&&A=\left\{ (x,y,z) \in R(f,4N) : y \equiv 2 \ (\text{mod} \ 4), xz \equiv 1 \ (\text{mod} \ 2) 
\right\} , \\
&&B=\left\{ (x,y,z) \in R(f,4N) : x \equiv 2 \ (\text{mod} \ 4), yz \equiv 1 \ (\text{mod} \ 2) 
\right\} .
\end{eqnarray*}
Note that
$$
r(f,4N)-r(f,N)=\vert A\vert +\vert B\vert.
$$
Thus it is sufficient to show $t(a,b,c;n)=\vert A\vert$ and $t(a,b,c;n)=\vert B\vert$.
To show the first equality, we apply  Lemma \ref{lem2} to show that
$$
r_{(1,1,1)}(f,N)=\displaystyle \sum_{y\in \z}r_{(1,1)}(ax^2+cz^2,N-by^2)=\displaystyle \sum_{y\in \z}r_{(1,1)}(ax^2+cz^2,4(N-by^2))=\vert A\vert .
$$
The proof of $t(a,b,c;n)=\vert B\vert$ is quite similar to this. This completes the proof.
\end{proof}


\begin{rmk}
All triples $(a,b,c)$ satisfying the assumption of Theorem \ref{thm1} are listed in Table 1 below. The triples marked with asterisk are exactly those that are listed in Conjecture 6.1 of \cite{s}. 

\begin{table}[ht]
\begin{tabular}{|l|}
\hline \rule[-2mm]{-2.5mm}{7mm} 
 \quad  $(1,1,7)^*, \  (1,1,15)^*, \ (3,3,5), \ (1,7,7)^*,  \ (3,5,5),  \  (1,7,15)^*, \  (1,9,15)^*$  \\  \hline

\rule[-2mm]{-2.5mm}{7mm} 
 \quad $(1,15,15)^*, \ (3,5,21), \  (1,7,49), \ (1,15,25)^*,\ (3,5,35), \ (3,5,45), \ (1,7,105)$ \\  \hline

\rule[-2mm]{-2.5mm}{7mm} 
 \quad $(3,5,75), \ (1,15,105),\ (3,21,35), \ (1,15,225), \ (9,15,25), \ (5,21,35), \ (7,15,105)$ \\  \hline

\end{tabular}
\vskip 1mm \caption{}
\end{table}
\end{rmk}

\begin{thm} \label{thm2} Let $a,b$ be relatively prime positive odd integers such that one of four fractions $\frac{b}{a}, \frac{a}{b},\frac{3a}{b},\frac{b}{3a}$ is contained in $\{ \frac5 3, 7, 15 \}$. 
 Then, for any positive integer $n$, we have 
$$
2t(a,3a,b;n)=3r(\langle a,3a,b\rangle,8n+4a+b)-r(\langle a,3a,b\rangle,4(8n+4a+b)).
$$
\end{thm}

\begin{proof}
Since all the other cases can be treated in a similar manner, we only consider the case when 
$\frac{b}{3a}=\frac{5}{3}$, that is, $(a,3a,b)=(1,3,5)$.
One may easily show that if $x^2+3y^2+5z^2=4(8n+9)$, then 
\[
\left( x^2,3y^2,5z^2 \right) \equiv (0,0,4), (1,3,0), (4,0,0), (4,3,5), \  \ \text{or} \ \ (4,4,4) \ (\text{mod} \ 8).
\]
Let 
$$
f=f(x,y,z)=x^2+3y^2+5z^2 \quad \text{and} \quad N=8n+9.
$$ 
From the above observation, we have
$$
\begin{array} {rl}
3r(f,N)-r(f,4N)\!\!\!&=3r_{(0,0,0)}(f,4N)-r(f,4N)\\[0.5em]
&=2r_{(0,0,0)}(f,4N)-r_{(1,1,0)}(f,4N)-r_{(0,1,1)}(f,4N).
\end{array}
$$
Therefore, it suffices to show that 
$$
2r_{(1,1,1)}(f,N)=2r_{(0,0,0)}(f,4N)-r_{(1,1,0)}(f,4N)-r_{(0,1,1)}(f,4N).
$$
Since $r_{(0,0,0)}(f,4N)=r(f,N)$ and 
$$
r(f,N)=r_{(1,1,1)}(f,N)+r_{(1,0,0)}(f,N)+r_{(0,0,1)}(f,N),
$$
it is enough to show that
$$
r_{(1,0,0)}(f,N)=\displaystyle \frac12 r_{(1,1,0)}(f,4N) \ \ \text{and} \ \  r_{(0,0,1)}(f,N)=\displaystyle \frac12r_{(0,1,1)}(f,4N).
$$
To prove the first assertion, we apply (i) of Lemma \ref{lem1} to show that
$$
\begin{array} {rl}
r_{(1,0,0)}(f,N)\!\!\!&=\displaystyle \sum_{z\in \z}r_{(1,0)}(x^2+3y^2,N-5z^2)\\[0.5em]
&=\displaystyle \frac12\sum_{z\in \z}r_{(1,1)}(x^2+3y^2,4(N-5z^2))=\displaystyle \frac12 r_{(1,1,0)}(f,4N).
\end{array}
$$
For the second assertion, we apply (iii) of Lemma \ref{lem1} and Lemma \ref{lem2} to show that
$$ 
\begin{array} {rl}
r_{(0,0,1)}(f,N)\!\!\!&=\displaystyle \sum_{z\in \z}r_{(0,0)}(x^2+3y^2,N-5z^2)\\[0.5em]
&=\displaystyle \frac12\sum_{z\in \z}r_{(1,1)}(x^2+3y^2,N-5z^2)=\displaystyle \frac12 r_{(1,1,1)}(x^2+3y^2+5z^2,N)\\[0.5em]
&=\displaystyle \frac12\sum_{x\in \z}r_{(1,1)}(3y^2+5z^2,N-x^2)=\displaystyle \frac12\sum_{x\in \z}r_{(1,1)}(3y^2+5z^2,4(N-x^2))\\[0.5em]
&=\displaystyle \frac12 r_{(0,1,1)}(f,4N).\\[0.5em]
\end{array}
$$
This completes the proof.
\end{proof}


\begin{rmk}
All triples $(a,3a,b)$ satisfying the assumption of the Theorem \ref{thm2} are listed in Table 2 below. Those triples marked with asterisk are exactly those that are listed in Conjecture 6.2 of \cite{s}. 

\begin{table}[ht]
\begin{tabular}{|l|}
\hline \rule[-2mm]{-2.5mm}{7mm} 
 \quad  $(1,3,5)^*,\  (1,3,7)^*,\ (1,3,15)^*, \ (1,3,21)^*,\ (1,5,15)^*,\ (1,3,45) $ \\  \hline

\rule[-2mm]{-2.5mm}{7mm} 
 \quad $(3,5,9)^*,\  (1,7,21)^*, \  (3,5,15)^*, \ (3,7,21)^*, \ (1,15,45), \ (5,9,15)  $ \\  \hline
\end{tabular}
\vskip 1mm \caption{}
\end{table}
\end{rmk}

\begin{thm} \label{thm3}
Let $(a,b,c)\in \{ (1,2,15), (1,15,18), (1,15,30)\}$. For any positive even integer $n$, we have
\begin{equation} \label{1215}
2t(a,b,c;n)=r(\langle a,b,c\rangle,4(8n+a+b+c))-r(\langle a,b,c\rangle,8n+a+b+c).
\end{equation}
\end{thm}

\begin{proof}
First, assume that $(a,b,c)=(1,2,15)$. 
Let
$$
f=f(x,y,z)=x^2+2y^2+15z^2 \quad \text{and} \quad N=8n+18.
$$
One may easily show that if $f(x,y,z)=4N$, then 
$$
\left( x^2,2y^2,15z^2 \right) \equiv (0,0,0), (1,0,7), \ \ \text{or} \ \ (4,0,4) \ (\text{mod} \ 8).
$$
Hence the right-hand side of Equation \eqref{1215} is 
$$
r(f,4N)-r(f,N)=r_{(1,0,1)}(f,4N).
$$
Note that
$$
\begin{array} {rl}
r_{(1,1,1)}(f,N)\!\!\!&=\displaystyle \sum_{y\in \z}r_{(1,1)}(x^2+15z^2,(N-2y^2)) \\[1em]
&=\displaystyle \sum_{y\in \z}r_{(1,1)}(x^2+15z^2,4(N-2y^2))
=r_{(1,1,1)}(x^2+8y^2+15z^2,4N) \\[1em]
&=\left\vert \left\{ (x,y,z) \in R(f,4N) : 
xz \equiv 1 (\text{mod} \ 2), \ y \equiv 2 (\text{mod} \ 4) \right\} \right\vert \\[1em]
\end{array}
$$
by Lemma \ref{lem2}. Since
$$
\left\vert \left\{ (x,y,z) \in R(f,4N) : 
xz \equiv 1 (\text{mod} \ 2), y \equiv 0 (\text{mod} \ 4) \right\} \right\vert =r(x^2+32y^2+15z^2,4N),
$$
it suffices to show that 
\begin{equation} \label{1215-1}
r_{(1,1,1)}(f,N)=r(x^2+32y^2+15z^2,4N).
\end{equation}
It is well known that 
$$
\gen(f_1=4x^2+4y^2+8z^2+2xy)=\{f_1,f_2,f_3\}, 
$$
where $f_2=4x^2+6y^2+6z^2+4yz+2xz+2xy$, $f_3=2x^2+6y^2+12z^2+6yz+2xz$, and
$$
\gen(g_1=4x^2+8y^2+18z^2+8yz+4xz)=\{g_1, g_2=2x^2+10y^2+24z^2\}.
$$
Note that 
$$
r_{(1,1,1)}(f,N)=r(x^2+2(x-2y)^2+15(x-2z)^2,N)=r(g_1,N).
$$
On the other hand, the right-hand side of Equation \eqref{1215-1} is
$$
\begin{array}{rl}
&r \left( x^2+15y^2+32z^2,4N\right)=r \left( (3x+y)^2+15(x+y)^2+32z^2,4N \right) \\[0.5em]
&=r \left( 12x^2+8y^2+16z^2+18xy,2N\right)\\[0.5em]
&=r \left( 48x^2+8y^2+16z^2+36xy,2N \right)
+r \left( 12x^2+32y^2+16z^2+36xy,2N \right) \\[0.5em]
&=2r \left(f_1,N \right) .
\end{array}
$$
Therefore, it suffices to show that for any positive even integer $n=2m$,
\begin{equation} \label{1215-2}
2r(f_1,16m+18)=r(g_1,16m+18).
\end{equation}
 By the Minkowski-Siegel formula, we have
\[
r(f_1,16m+18)+2r(f_2,16m+18)+r(f_3,16m+18)=r(g_1,16m+18)+r(g_2,16m+18).
\]
If $f_1(x,y,z)=16m+18$, then  
 one may easily check that $x+3y-4z \equiv 0 \Mod 8$, 
 and  if $f_2(x,y,z)=16m+18$, then $x-6y+2z \equiv 0 \Mod 8$. 
If we define a map 
$$
\begin{array}{rl}
\phi_1:\!\!\!&\left\{(x,y,z) \in R(f_1,16m+18) :
x+3y-4z \equiv 0 \text{ (mod 16)} \right\} \\[0.5em]
&\hskip 2pc \to \left\{(x,y,z) \in R(f_2,16m+18) :
x-6y+2z \equiv 0 \text{ (mod 16)}\right\}
\end{array}
$$
by $\phi_1(x,y,z)=\displaystyle\left(\frac{12x+4y+16z}{16},\frac{-11x-y+12z}{16},
\frac{x-13y-4z}{16} \right)$, then it is a bijective map. Furthermore,  the map 
$$
\begin{array}{rl}
\phi_2:\!\!\!&\left\{(x,y,z) \in R(f_1,16m+18) :  
x+3y-4z \equiv 8 \text{ (mod 16)} \right\} \\[0.5em]
&\hskip 2pc \to \left\{(x,y,z) \in R(f_2,16m+18) : 
x-6y+2z \equiv 8 \text{ (mod 16)} \right\}
\end{array}
$$
defined by $\phi_2(x,y,z)=\displaystyle\left(\frac{4x+12y-16z}{16},\frac{-13x+y+4z}{16},\frac{-
x-11y-12z}{16} \right)$ is also bijective. Therefore,  we have
\begin{equation} \label{1215-3}
r(f_1,16m+18)=r(f_2,16m+18). 
\end{equation}
Note that the above equation does not hold, in general, if $n$ is odd. If we define two maps 
$$
\phi_3 : R(\langle 8,10,24\rangle,16m+18) \to
R(f_1,16m+18) \ \  \text{by}  \ \  \phi_3(x,y,z)=(y+2z,y-2z,x)
$$ 
and 
$$
\phi_4 :
R(\langle 2,24,40\rangle,16m+18) \to R(f_3,16m+18) \  \  
\text{by}   \  \ \phi_4(x,y,z)=(x+z,2y+z,-2z),
$$ 
then one may easily check that both of them are bijective. Hence we have
$$
\begin{array}{rl}
r(g_2,16m+18)\!\!&\!\!\!=r \left(\langle 8,10,24\rangle,16m+18 \right)
+r \left(\langle 2,24,40\rangle,16m+18 \right) \\
&\!\!\!=r(f_1,16m+18)+r(f_3,16m+18),
\end{array}
$$
for any non negative integer $m$. Therefore, from the Minkowski-Siegel formula given above, we have $2r(f_2,16m+18)=r(g_1,16m+18)$ for any non negative integer $m$. Equation \eqref{1215-2} follows directly from this and Equation \eqref{1215-3}. 

For the other two cases,  one may easily show Equation \eqref{1215} by replacing  $N,f_i,g_i$ and $\phi_i$ with the following data:

\noindent {\bf (1) $(a,b,c)=(1,15,18)$.} In this case, we let $N=8n+34$ and 
$$
\begin{array}{l}
f_1=4x^2+4y^2+72z^2+2xy, \\
f_2=4x^2+16y^2+22z^2+14yz-2xz+4xy, \\
f_3=6x^2+16y^2+16z^2-8yz+6xz+6xy, \\
\end{array}
$$
and
$$
g_1=4x^2+34y^2+34z^2+8yz+4xz+4xy, \quad  g_2=10x^2+18y^2+24z^2.
$$
Define 
$$
\begin{array}{rl}
\phi_1 : &\left\{ (x,y,z)\in R(f_1,16m+34) : 3x+y+4z \equiv 0 \ (\text{mod } 16)\right\} \\ [0.5em]
&\hskip 2pc \to \left\{ (x,y,z) \in R(f_2,16m+34) : 3x-y+2z \equiv 0 \ (\text{mod } 16)\right\}
\end{array}
$$
by
$$
\phi_1(x,y,z)=\left(\displaystyle \frac{x-5y-68z}{16},\frac{-5x-7y+20z}{16},\frac{-4x+4y-16z}{16} \right),
$$
$$
\begin{array}{rl}
\phi_2 : &\left\{ (x,y,z)\in R(f_1,16m+34) : 3x+y+4z \equiv 8 \ (\text{mod } 16)\right\} \\ [0.5em]
&\hskip 2pc  \to \left\{ (x,y,z) \in R(f_2,16m+34) : 3x-y+2z \equiv 8 \ (\text{mod } 16)\right\}
\end{array}
$$
by
$$
\phi_2(x,y,z)=\left(\frac{9x-5y-52z}{16},\frac{3x+9y+4z}{16},\frac{4x-4y+16z}{16} \right),
$$
and
$$
\begin{array} {rl} 
&\phi_3 : R(10x^2+24y^2+72z^2,16m+34) \\ [0.5em] 
&\hskip 6pc \to R(f_1,16m+34) \ \ 
\text{by} \ \ 
\phi_3(x,y,z)=(x-2y,x+2y,z),
\end{array}
$$
$$
\begin{array} {rl} 
&\phi_4 : R(18x^2+24y^2+40z^2,16m+34)\\ [0.5em] 
&\hskip 4pc  \to R(f_3,16m+34)
 \ \ \text{by} \ \ 
\phi_4(x,y,z)=(x+2y,-x+z,-x-z).
\end{array} 
$$

\noindent {\bf (2) $(a,b,c)=(1,15,30)$.}   In this case, we let $N=8n+46$ and 
$$
\begin{array}{l}
f_1=4x^2+4y^2+120z^2+2xy, \\
f_2=4x^2+16y^2+34z^2+14yz-2xz+4xy, \\
f_3=10x^2+16y^2+16z^2+8yz+10xz+10xy, \\
\end{array}
$$
and
$$
g_1=4x^2+46y^2+46z^2+32yz+4xz+4xy, \quad g_2=6x^2+30y^2+40z^2.
$$
Define 
$$
\begin{array}{rl}
\phi_1 : &\left\{ (x,y,z)\in R(f_1,16m+46) : 3x-y-4z \equiv 0 \ (\text{mod } 16)\right\} \\ [0.5em]
&\hskip 2pc \to \left\{ (x,y,z) \in R(f_2,16m+46) : 3x-y+2z \equiv 8 \ (\text{mod } 16)\right\}
\end{array}
$$
by
$$
\phi_1(x,y,z)=\left(\displaystyle \frac{7x-13y-4z}{16},\frac{-3x+y-44z}{16},\frac{-4x-4y+16z}{16} \right),
$$
$$
\begin{array}{rl}
\phi_2 : &\left\{ (x,y,z)\in R(f_1,16m+46) : 3x-y-4z \equiv 8\  (\text{mod } 16)\right\} \\  [0.5em]
&\hskip 2pc  \to \left\{ (x,y,z) \in R(f_2,16m+46) : 3x-y+2z \equiv 0 \ (\text{mod } 16)\right\}
\end{array}
$$
by
$$
\phi_2(x,y,z)=\left(\frac{9x-11y+20z}{16},\frac{3x+7y+28z}{16},\frac{-4x-4y+16z}{16} \right),
$$
and
$$
\begin{array} {rl}
&\phi_3 : R(6x^2+40y^2+120z^2,16m+46) \\ [0.5em]
&\hskip 7pc \to R(f_1,16m+46)
 \ \ 
\text{by}
 \ \ 
\phi_3(x,y,z)=(x+2y,-x+2y,z),
\end{array}
$$
$$
\begin{array} {rl}
&\phi_4 : R(24x^2+30y^2+40z^2,16m+46)\\ [0.5em]
&\hskip 6pc \to R(f_3,16m+46)
\ \ 
\text{by} \ \ 
\phi_4(x,y,z)=(-y-2z,x+y,-x+y).
\end{array} 
$$
This completes the proof. \end{proof}


\begin{thm} \label{thm4}
For any positive integer $n$ such that $n \not\equiv 1 \Mod{3}$, we   have
\begin{equation} \label{1127}
2t(1,1,27;n)=r(x^2+y^2+27z^2,4(8n+29))-r(x^2+y^2+27z^2,8n+29).
\end{equation}
\end{thm}

\begin{proof}
 Let $N=8n+29$ and 
$$
\begin{array}{l}
f=f(x,y,z)=x^2+y^2+27z^2, \\
g=g(x,y,z)=8x^2+20y^2+29z^2+4yz+8xz+8xy, \\
h=h(x,y,z)=2x^2+5y^2+27z^2+2xy.
\end{array}
$$
For any positive integer $m \not\equiv 1 \Mod 3$, we let 
$$
\delta_m=
\begin{cases}
1 & \text{if} \ \ m\equiv 0\ (\text{mod } 3), \\
2 & \text{if} \ \ m\equiv 2\ (\text{mod } 3).
\end{cases}
$$
Note that
\begin{equation} \label{1127-1}
r(f,m)=\delta_m \left\vert \left\{ (x,y,z) \in R(f,m) : x \equiv y \Mod 3 \right\} \right\vert .
\end{equation}
Since
$$
r(f,4N)=\delta_N \cdot r(x^2+(x-3y)^2+27z^2,4N)=\delta_N \cdot r(h,4N)
$$
and
$$
\begin{array}{l}
\left\vert \left\{ (x,y,z) \in R(f,4N) : y \equiv 0 \Mod 2 \right\} \right\vert \\[0.5em] 
=\delta_N \cdot r(x^2+4(x-3y)^2+27z^2,4N)=\delta_N \cdot r(8x^2+5y^2+27z^2+4xy,4N) \\[0.5em]
=\delta_N \left\vert \left\{ (x,y,z) \in R(h,4N) : x \equiv 0 \Mod 2 \right\} \right\vert,
\end{array}
$$
we have
\begin{equation} \label{1127-2} 
\vert \left\{ (x,y,z) \in R(f,4N) : \text{$y$ is odd} \right\} \vert=\delta_N\vert \left\{ (x,y,z) \in R(h,4N) :  \text{$x$ is odd}  \right\}\vert.
\end{equation}
One may easily show that if $(x,y,z) \in R(f,4N)$, then
\[
\left( x^2,y^2,27z^2 \right) \equiv (0,0,4), (0,1,3), (0,4,0), (1,0,3), (4,0,0), (4,4,4) \Mod 8.
\]
From this and Equation \eqref{1127-2}, the right hand side of Equation \eqref{1127} becomes
$$
r(f,4N)-r(f,N)=2\delta_N \left\vert \left\{ (x,y,z) \in R(h,4N) : x \equiv
1 \Mod 2 \right\} \right\vert.
$$
On the other hand, by Equation \eqref{1127-1},
$$
\begin{array}{rl}
t(1,1,27;n)\!\!\!&=r_{(1,1,1)}(f,N) \\[0.5em]
&=\delta_N \left\vert \left\{ (x,y,z) \in R(f,N) : x\equiv y \Mod 3,\ x\equiv y\equiv z \Mod 2 \right\} \right\vert \\[0.5em]
&=\delta_N \cdot r(x^2+(x-6y)^2+27(x-2z)^2,N)=\delta_N \cdot r(g,N).
\end{array}
$$
Therefore, it is enough to show that 
$$
r(g,N)=\left\vert \left\{ (x,y,z) \in R(h,4N) : \ x \equiv 
1 \ (\text{mod } 2) \right\} \right\vert .
$$

Now, we let
\begin{eqnarray*}
&&A=\left\{ (x,y,z) \in R(g,N) : x \equiv 0 \Mod 2
\right\}, \\
&&B=\left\{ (x,y,z) \in R(h,4N) : \ x \equiv 1 \ (\text{mod} \ 2), x+z \equiv 0 \Mod 8 \right\}.
\end{eqnarray*}
Note that $x+z \equiv 8 \Mod {16}$ if $(x,y,z) \in B$.
Define a map $\phi : A \to B$ by
\[
\phi(x,y,z)=\left( x-7z,\ -x-4y+z,\ -x-z \right).
\]
Then, one may easily show that $\phi$ is a bijection.
Since $g(x+z,y,-z)=g(x,y,z)$ and  $z_0$ is odd for any $(x_0,y_0,z_0) \in R(g,N)$,  we have
$$
\left\vert \left\{ (x,y,z) \in R(g,N) : x \equiv 0 \Mod 2 
\right\} \right\vert 
=\left\vert \left\{ (x,y,z) \in R(g,N) : x \equiv 1 \Mod 2
\right\} \right\vert
$$
and thus
$$
r(g,N)=2\left\vert \left\{ (x,y,z) \in R(g,N) : x \equiv 0 \Mod 2
\right\} \right\vert.
$$

Now, we are ready to prove the assertion. Note that if 
$(x,y,z) \in R(h,4N)$ and $x \equiv 1 \Mod 2$, 
then $z \equiv \pm x \Mod 8 $. Therefore, we have 
$$
\begin{array}{l}
\left\vert \left\{ (x,y,z) \in R(h,4N) : \ x \equiv 1 \Mod 2 \right\} \right\vert \\[0.5em]
=2 \left\vert \left\{ (x,y,z) \in R(h,4N) : \ x \equiv 1 \Mod 2, \ x+z \equiv 0 \Mod 8 \right\} 
\right\vert \\[0.5em]
=2\vert B \vert =2\vert A\vert =r(g,N).
\end{array}
$$
\\
This completes the proof.
\end{proof}

Finally, we prove the Conjecture 6.7 in \cite{s}. 

\begin{thm} \label{thm5}
For a positive integer $n$, the Diophantine equation 
$$
\mathcal T_{(1,1,6)}(x,y,z)=\displaystyle \frac{x(x-1)}{2}+\frac{y(y-1)}{2}+6\frac{z(z-1)}{2}=n
$$ 
has an integer solution if and only if $n \not\equiv 2 \cdot 3^{2r-1}-1 \Mod{3^{2r}}$ for any positive integer $r$.
\end{thm}

\begin{proof}
Note that $\mathcal T_{(1,1,6)}(x,y,z)=n$ has an integer solution if and only if $f(x,y,z)=x^2+y^2+6z^2=8n+8$ has an integer 
solution $x,y,z$ such that $xyz \equiv 1 \Mod 2$. Since the ternary quadratic form $f(x,y,z)$ has class number one, it represents every integer that is locally represented (see 102.5 of \cite{OM}).  
Therefore, one may easily check that $f(x,y,z)=8n+8$ has an integer solution if and 
only if   $n \not\equiv 2 \cdot 3^{2r-1}-1 \Mod{3^{2r}}$ for any positive integer $r$.

Now,  assume that $n$ is a  positive integer such that $n \not\equiv 2 \cdot 3^{2r-1}-1 \Mod{3^{2r}}$ for any positive integer $r$. 
Note that $f(x,y,z)=8n+8$ has an integer solution $x,y,z$ such that $xyz \equiv 1 \Mod 2$ if and 
only if $r(f,8n+8)-r(f,2n+2)>0$. 
By the Minkowski-Siegel formula, we have 
$$
\frac {r(f,8n+8)}{r(f,2n+2)}=2\frac {\alpha_2(f,8n+8)}{\alpha_2(f,2n+2)},
$$
where $\alpha_2$ is the local density over $\z_2$ (for details, see, for example, \cite{Ki}). For a positive integer $s$ and a positive odd integer $t$, one may easily compute by using the result of \cite{ya} that 
$$
\alpha_2(f,2^st)=\begin{cases} 
2-3\cdot 2^{-s/2} & \text{if} \quad s \equiv 0 \Mod{2}, \\ 
2-2^{(1-s)/2} & \text{if} \quad s \equiv 1 \Mod{2}, \ t \equiv 1 \Mod{8}, \\
2 & \text{if} \quad s \equiv 1 \Mod{2}, \ t \equiv 5 \Mod{8}, \\
2-3\cdot 2^{(-s-1)/2} & \text{if} \quad s \equiv 1 \Mod{2}, \ t \equiv 3 \ \text{or} \ 7 \Mod{8}.
\end{cases} 
$$
Therefore, we have 
$2\alpha_2(f,8n+8)>\alpha_2(f,2n+2)$ for any positive integer $n$.
This completes the proof. 
\end{proof}



\begin{thebibliography}{abcd}

\bibitem {ach} C. Adiga, S. Cooper and J. H. Han, {\em A general relation between sums of squares and sums of triangular numbers}, Int. J. Number Theory \textbf{1}(2005), 175-182. 

\bibitem {bch} N. D. Baruah, S. Cooper and M. Hirschhorn, {\em Sums of squares and sums of triangular numbers induced by partitions of 8}, Int. J. Number Theory \textbf{4}(2008) 525-538.

\bibitem {Ki} Y. Kitaoka, {\em Arithmetic of quadratic forms}, Cambridge University Press, 1993.


\bibitem {OM} O. T. O'Meara, {\em Introduction to quadratic forms}, Springer-Verlag, 1973.

\bibitem {s0} Z. H. Sun, {\em Some relations between $t(a,b,c,d;n)$ and $N(a,b,c,d;n)$}, Acta Arith. \textbf{175}(2016), 269-289.

\bibitem {s} Z. H. Sun,  {\em Ramanujan's theta functions and sums of triangular numbers}, preprint. 

\bibitem {ya} T.  Yang, {\em An explicit formula for local densities of quadratic forms}, J. Number Theory \textbf {72}(1998), 309-356.


\bibitem {y} X. M. Yao, {\em The relation between  $N(a,b,c,d;n)$ and $t(a,b,c,d;n)$ and $(p,k)$-parametrization of theta functions}, J. Math. Anal. Appl. \textbf{453}(2017), 125-143. 

\end{thebibliography}
\end{document}